\documentclass[10pt]{amsart}
\usepackage[utf8]{inputenc}
\usepackage{amsmath}
\usepackage{amssymb}
\usepackage{amsfonts}
\usepackage{hyperref}
\usepackage{url}
\usepackage{tikz}
\usepackage{float}
\usepackage{amsthm}
\usepackage{subcaption}
\newtheorem{theorem}{Theorem}

\newtheorem{lemma}[theorem]{Lemma}

\newtheorem{corollary}[theorem]{Corollary}

\newtheorem{remark}[theorem]{Remark}

\title{A look at generalized perfect shuffles}
\author[S. Johnson, L. Manny, C.~Van Cott, Q. Zhang]{Samuel Johnson, Lakshman Manny, \\Cornelia A. Van Cott, QiYu Zhang}

\begin{document}
\maketitle
\begin{abstract}
 Standard perfect shuffles involve splitting a deck of $2n$ cards into two stacks and interlacing the cards from the stacks. There are two ways that this interlacing can be done, commonly referred to as an in shuffle and an out shuffle, respectively. In 1983, Diaconis, Graham, and Kantor determined the permutation group generated by in and out shuffles on a deck of $2n$ cards for all $n$. Diaconis et al.~concluded their work by asking whether similar results can be found for so-called generalized perfect shuffles. For these new shuffles, we split a deck of $mn$ cards into $m$ stacks and similarly interlace the cards with an in $m$-shuffle or out $m$-shuffle (denoted $I_m$ and $O_m$, respectively). In this paper, we find the structure of the group generated by these two shuffles for a deck of $m^k$ cards, together with $m^y$-shuffles, for all possible values of $m$, $k$, and $y$. The group structure is completely determined by $k/\gcd(y,k)$ and the parity of $y/\gcd(y,k)$. In particular, the group structure is independent of the value of $m$. \\   \end{abstract}

  
 
 \section{Introduction }\label{introduction}

When handed an unshuffled deck of cards, most of us respond in the same way. We attempt (however feebly) to do a so-called {\em perfect shuffle}. A perfect shuffle splits a deck of cards into two equal stacks and then perfectly interlaces the cards from the two stacks one after the other. Only experienced gamblers and magicians can perform perfect shuffles reliably, and yet the mathematics behind perfect shuffles has a rich history, including everything from mathematical card tricks to sophisticated research.

\begin{figure}
\begin{center}
  \includegraphics[width=7cm]{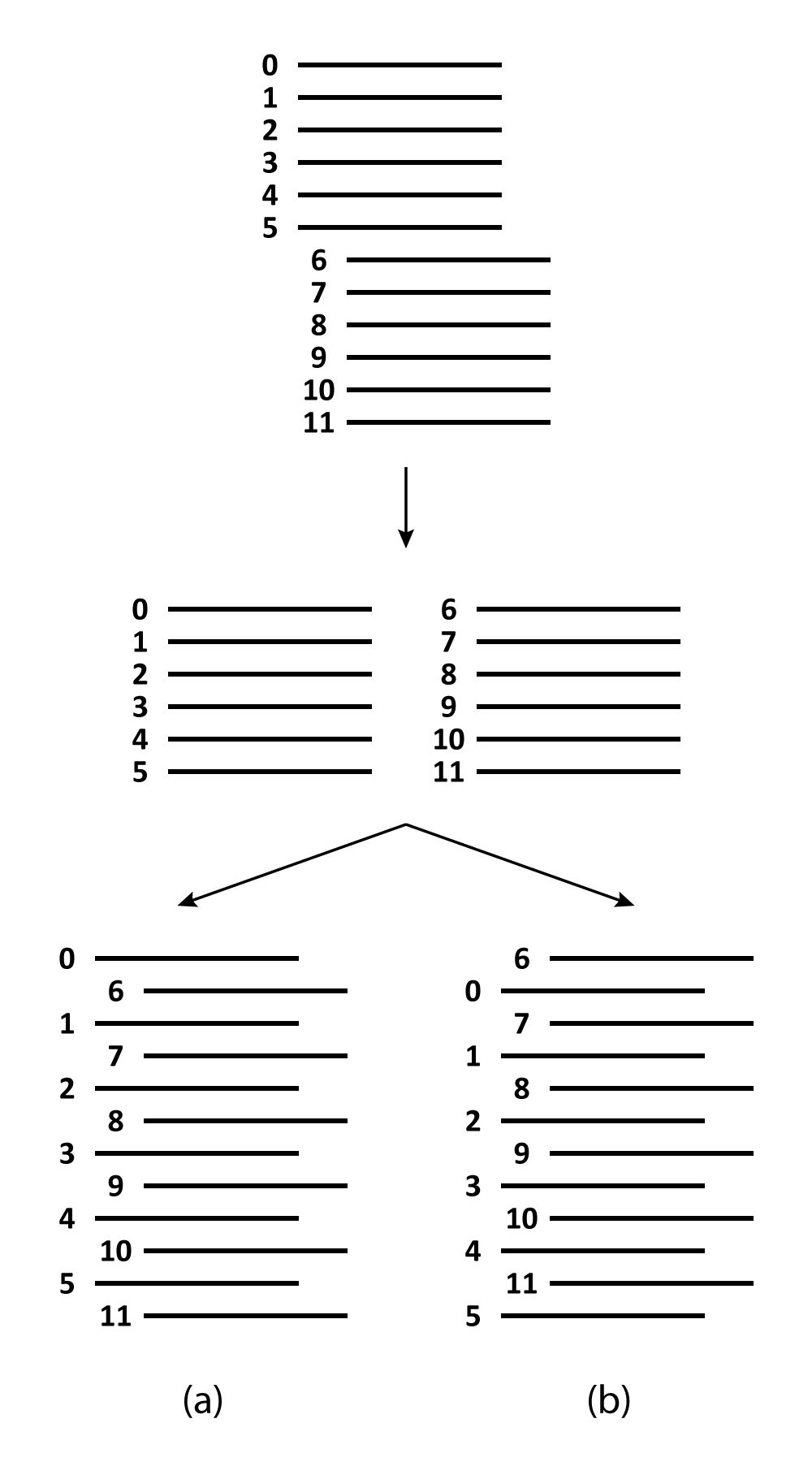}
  \end{center}
    \caption{Starting with a deck of 12 cards, we split the deck and perform (a) an out shuffle and (b) an in shuffle. }~\label{perfectshuffle}
\end{figure}

There are two ways that the interlacing of a perfect shuffle can be done. The {\em out shuffle} (denoted $O$) is the instance in which the original top card remains at the top, while the {\em in shuffle} (denoted $I$)  is the instance in which the original top card becomes the second card, lying underneath the top card of the second stack. See Figure~\ref{perfectshuffle} for an example. Perfect shuffles are permutations of the set of $2n$ cards, so the group generated  by the in shuffle and out shuffle $\langle I, O\rangle$ is a subgroup of the symmetric group $S_{2n}$.  In 1983, Persi Diaconis, Ron Graham, and William Kantor determined the group $\langle I, O\rangle$ for a deck of $2n$ cards for all $n$~\cite{Diaconis}. (The development of this paper by Diaconis et al.~was nicely described in narrative form in {\em Science}~\cite{Kolata}.) 

One might expect that the order of these shuffle groups $\langle I, O\rangle$ would grow with the size of the deck, but this is not the case. Compare, for example, decks with 30 cards and 32 cards, respectively. With 30 cards, the group generated by in and out shuffles has order $15!\cdot 2^{14}$, which is over 21 quadrillion. And yet, with 2 more cards (32 cards total), the shuffle group has order $160$. (As we will see later, when the deck size is a power of 2, the shuffle group is always relatively small.)

Several colorful variations on perfect shuffles have since been studied, including flip shuffles, horseshoe shuffles, Monge shuffles, and milk shuffles, to name a few~\cite{Bayer, Butler, Diaconis, Ledet}. At the conclusion of their paper, Diaconis et al. mentioned another natural variation -- the so-called {\em generalized perfect shuffles}. We describe these shuffles here.

Suppose one has a deck of $mn$ cards. A generalized perfect shuffle (also called an $m$-shuffle) is performed first by dividing the cards into $m$ equal stacks (the first stack contains the first $n$ cards, the second stack contains the next $n$ cards, and so on). Place these $m$ stacks side by side in order from left to right.  A priori, there are $m!$ different patterns which one can use to interlace the cards from these $m$ stacks to reassemble the deck. We consider two interlacing patterns which directly generalize the standard in and out perfect shuffles: 
\begin{itemize}
\item {\bf Out $m$-shuffle} (denoted $O_{m}$). Pick up the cards from the $m$ stacks from {\em left to right}. In this case, the original top card will remain the top card when the shuffle is complete.
\item {\bf In $m$-shuffle} (denoted $I_m$). Pick up the cards from the  $m$ stacks from {\em right to left}. In this case, the original top card will now be the $m^{th}$ card when the shuffle is complete.
\end{itemize}
See Figure~\ref{generalizedshuffle} for examples of 3-shuffles.  With this perspective, the standard perfect shuffles studied by Diaconis et al.~are a special case, being {\em 2-shuffles}. 

A variety of results have already been discovered about generalized perfect shuffles (or, $m$-shuffles, as we will call them). The group generated by an $m$-shuffle and the simple cut was investigated by Morris and Hartwig~\cite{MorrisHartwig}. They found the group in most cases to be either the full symmetric group of the deck or the alternating group of the deck. 
Working even more generally, others have considered allowing all $m!$ interlacing patterns of the $m$ stacks rather than restricting to the in and out $m$-shuffles. The groups generated in this broader context have been determined for several infinite families~\cite{Amarra, Medvedoff}.

Here in this work, we return to the original question posed by Diaconis et al.~\cite{Diaconis} and which was posed again in~\cite{Medvedoff}. Namely, we seek to find the group generated by in and out $m$-shuffles $\langle I_m,O_m \rangle$. The solution to the analogous question for the group $\langle I,O \rangle$ of standard perfect shuffles was a mathematical and computational tour de force requiring many cases. We expect the situation with generalized shuffles to be similarly complex. Here in this paper, we restrict to decks of size $m^k$ for some $m, k >1$. With this card deck, we consider $m^y$-shuffles for some $y<k$. 

A natural starting place is to consider 2-shuffles on decks of size $2^k$. This special case fits into the setting of standard perfect shuffles and was studied by Diaconis et al. They proved the following.

\begin{theorem}\cite{Diaconis}\label{standardshuffles}
Consider a deck of $2^k$ cards for some $k\geq1$. The 2-shuffle group $\langle I_2,O_2 \rangle$ has order $2^kk$ and is isomorphic to $\mathbb{Z}^k_2 \rtimes \mathbb{Z}_k$, where $\mathbb{Z}_k$ acts on $\mathbb{Z}^k_2$  by a cyclic shift.
\end{theorem}

Broadening this result, we find the structure of the $m^y$-shuffle group on a deck of $m^k$ cards, for all possible values of $m, k$, and $y$. The result nicely generalizes the previous special case. 

\begin{theorem}\label{mainresult}
Consider a deck of $m^k$ cards for some $m, k>1$. Let $y$ be a positive integer such that $y<k$ and $\gcd[y,k] = c$. The group generated by in and out $m^y$-shuffles on this deck $ \langle I_{m^y},O_{m^y} \rangle$ can be described as follows.
 
\begin{enumerate}
\item  
If $y/c$ is odd, then the shuffle group is isomorphic to $(\mathbb{Z}_2)^{k/c} \rtimes \mathbb{Z}_{k/c}$, where $\mathbb{Z}_{k/c}$ acts by a cyclic shift on $(\mathbb{Z}_2)^{k/c}$. 

\item If $y/c$ is even, then the shuffle group is isomorphic to $(\mathbb{Z}_2)^{\frac{k}{c}-1} \rtimes \mathbb{Z}_{k/c}$, where $\mathbb{Z}_{k/c}$ acts on on $(\mathbb{Z}_2)^{\frac{k}{c}-1}$ as follows:
$$\phi(1)\cdot(a_1, a_2, \ldots, a_{\frac{k}{c} -1}) = (a_{\frac{k}{c} -1},~ a_1+ a_{\frac{k}{c} -1},~ a_2+ a_{\frac{k}{c} -1},~ \ldots,~ a_{\frac{k}{c} -2} + a_{\frac{k}{c} -1}).$$

\end{enumerate}
\end{theorem}
Certainly a few examples will clarify this result, and so we discuss examples throughout. Note that the order and the structure of these shuffle groups in Theorem~\ref{mainresult} are independent of the value of $m$. 


\begin{figure}
\begin{center}
    \includegraphics[width=7cm]{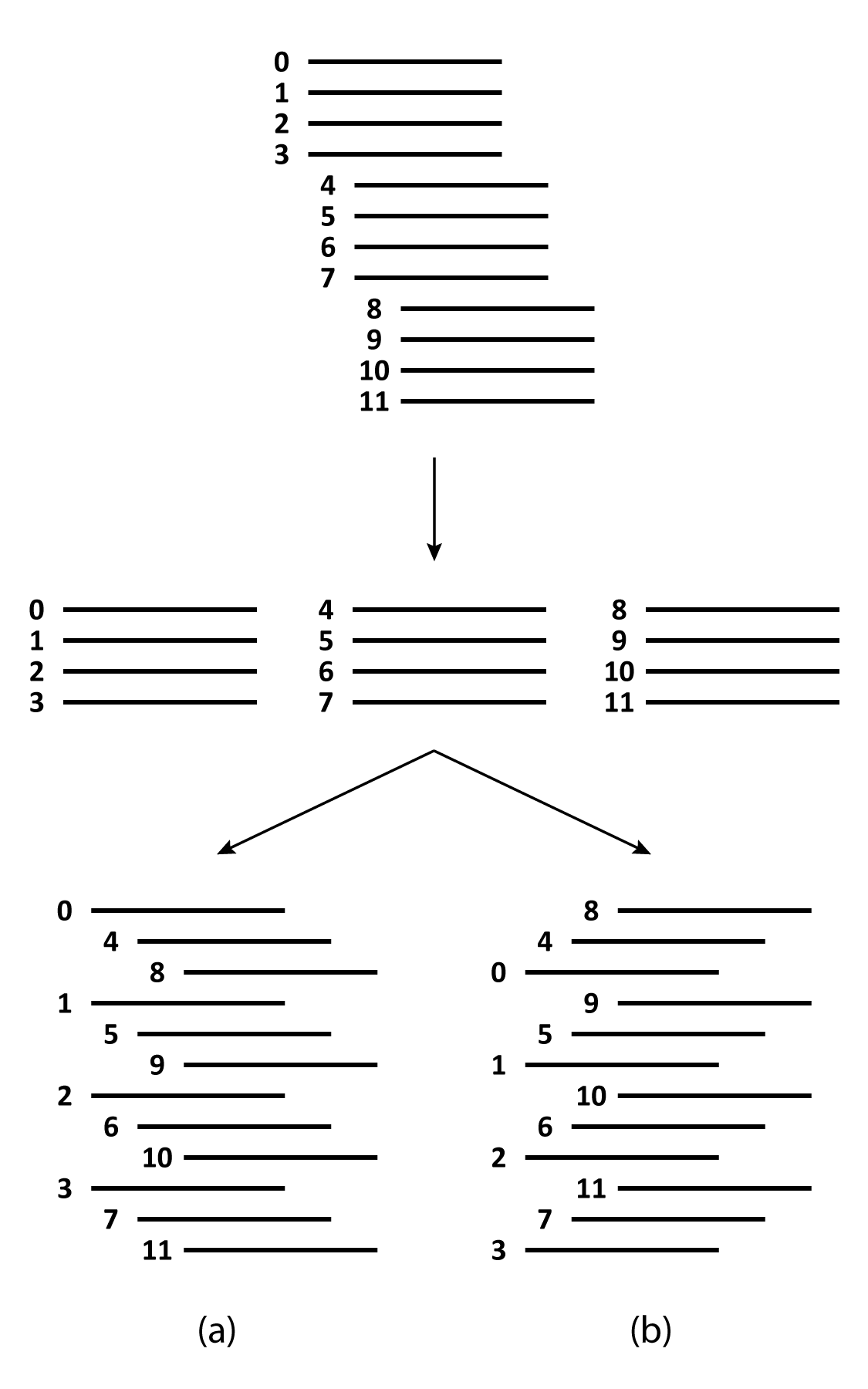}
    \end{center}
    \caption{Starting with a deck of 12 cards, we divide the deck into 3 stacks and perform generalized 3-shuffles: (a) an out 3-shuffle and (b) an in 3-shuffle. } \label{generalizedshuffle}
\end{figure}

\section{Standard perfect shuffles}\label{previousresults}
We begin by reviewing some of the beautiful mathematics underlying standard perfect shuffles (see~\cite{Diaconis} for the proofs). We only fit in a sampling of results here; further interesting patterns (and mathematical card tricks to impress your friends) are plentiful and provide ample opportunity for reading and investigation~\cite{DiaconisGraham, Ensley, Morris2, Ramnath, Rosenthal}.

Suppose we have a deck of $2n$ cards. Index each card by its distance from the top of the deck from $0$ to $2n-1$. 
Simple formulas give the location of each card after performing a standard perfect shuffle. The out shuffle fixes the cards with index $0$ and $2n-1$. And, for any card with index $i < 2n-1$,  the card's position after an out shuffle $O$ is as follows:
$$ O: i \longrightarrow 2i\!\!\!\pmod {2n - 1}$$
It follows, then, that the order of the out shuffle is given by the order of $2$ in $\mathbb{Z}_{2n-1}$.
An in shuffle $I$  moves the card with index $i$ as follows:
$$I:  i \longrightarrow 2i + 1\!\!\!\pmod {2n + 1}$$
From this, one can show that the order of the in shuffle is the order of $2$ in $\mathbb{Z}_{2n+1}$. 

For example, consider a deck of 52 cards. One needs to do only 8 out shuffles to bring the deck back to its original order, since 2 has order 8 in $\mathbb{Z}_{51}$. On the other hand, one must do 52 in shuffles before the deck returns to its original order, as 2 has order 52 in $\mathbb{Z}_{53}$.

Perfect shuffles preserve several types of symmetry in a deck of cards. Consider a pair of cards which are each located the same distance from the center of the deck. After an in or out shuffle, these two cards will both move to positions which are again the same distance from the center of the deck. Because of this, we say that in and out shuffles {\em preserve central symmetry}. 

The set of all permutations in $S_{2n}$ that preserve central symmetry form a subgroup of $S_{2n}$, which we denote by $B_n$. Observe that $B_n$ has order $n!~\!2^n$, and so it follows that the shuffle group $\langle I,O \rangle$ on $2n$ cards has order at most $n!~\!2^n$. In fact, the group $\langle I,O \rangle$ is isomorphic to $B_n$ if and only if $n\equiv 2 \pmod{4}$ and $n>6$. A deck of 52 cards $(n = 26$) is one example of this. Save for a few exceptional cases, the remaining shuffle groups $\langle I,O \rangle$ can be described as a kernel (or an intersection of kernels) of different homomorphisms from $B_n$ to $\mathbb{Z}_2$ (see~\cite{Diaconis} for details). The significant exceptional case is when the number of cards is a power of 2 (say, $2^k$). In this case, the shuffle group has order $2^kk$ and has structure $\mathbb{Z}_2^k \rtimes \mathbb{Z}_{k}$, as stated in Theorem~\ref{standardshuffles}.

Let's consider a specific example to illustrate Theorem~\ref{standardshuffles}. Begin with just 4 cards. A priori, we might be able to reach any of $24$ different card arrangements by shuffling these 4 cards. But Theorem~\ref{standardshuffles} tells us that we only reach 8 card arrangements with standard perfect shuffles, and the shuffle group is isomorphic to $(\mathbb{Z}_2 \times \mathbb{Z}_2) \rtimes \mathbb{Z}_2$. (This group is more commonly recognized as the dihedral group of a square, $D_4$.) 

This group's Cayley graph nicely illustrates the relationships among the card arrangements and shuffles. Let us consider each of these 8 possible card arrangements to be vertices on a graph. If one card arrangement can be obtained from another by an in shuffle or out shuffle, we draw a directed edge between the two associated vertices. In this example, the resulting graph is a cube. See Figure~\ref{FourCards}.

\begin{figure}\
\begin{center}
\includegraphics[width=10cm]{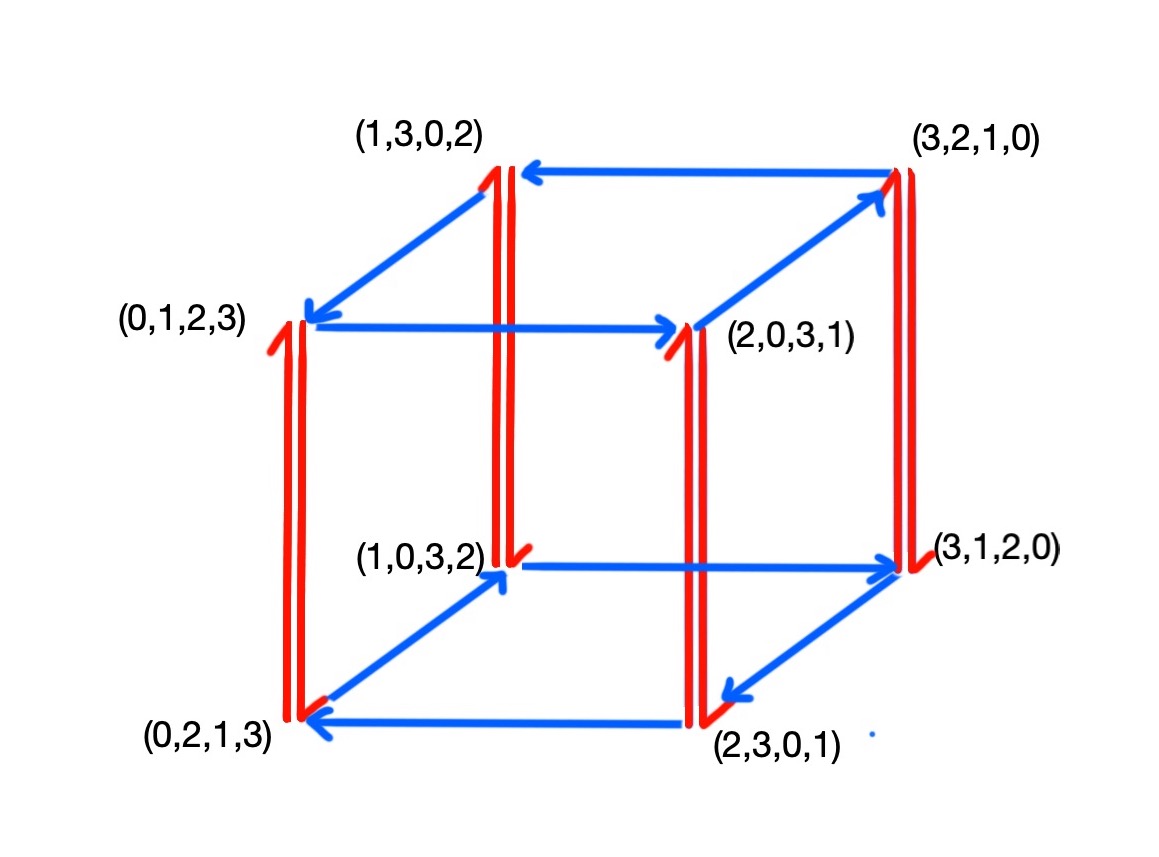}
\end{center}
\caption{Four cards (numbered in original order as 0, 1, 2, 3) can be shuffled with out shuffles (red) and in shuffles (blue) to reach eight different possible orderings total.}\label{FourCards}
\end{figure}

\section{Generalized perfect shuffles}\label{generalized}
Now we expand our perspective to consider a deck of $mn$ cards, together with in and out $m$-shuffles (denoted $I_m$ and $O_m$, respectively). We begin by studying the basic properties of these shuffles. The following lemma (also proved in~\cite{Medvedoff}) provides the critical information of where cards end up after an $m$-shuffle.\\

\begin{lemma}~\label{out}
Consider a deck of $mn$ cards. Index each card in the deck by its distance from the top of the deck from $0$ to $mn-1$.
\begin{enumerate}
\item  The out $m$-shuffle $O_m$ fixes the cards with index 0 and $mn-1$, and for $i < mn-1$, the index of the $i^{th}$ card after an out $m$-shuffle is:
$$O_m: i \longrightarrow  mi \pmod{mn - 1}$$
As such, the order of $O_m$ is the order of $m$ in $\mathbb{Z} _{mn-1}.$
\item The index of the $i^{th}$ card of the deck after an in $m$-shuffle $I_m$ is:
$$I_m: i  \longrightarrow  mi + (m-1) \pmod{mn + 1}$$
Thus the order of $I_m$ is the order of $m$ in $\mathbb{Z}_{mn+1}$.
\end{enumerate}
\end{lemma}
\begin{proof}
Starting with the deck of cards indexed $0$ to $mn-1$, we divide the cards into $m$ equal stacks and place the stacks in order side by side from left to right. The cards are then in the following array:

\begin{align}~\label{array}
\begin{array}{cccccccccc} 0 &&& n &&& 2n&&  \cdots & (m-1)n\\ 1 && &n+1&& &2n+1 && \cdots & \vdots \\ 2 & &&n+2&& & 2n+2 &&  \cdots & \vdots \\ \vdots  && & \vdots &&& \vdots & &\vdots & \vdots \\n-1 && & 2n-1 &&& 3n-1 && \cdots & mn-1 
\end{array}
\end{align}

Consider the card lying in row $j$ and column $p$, for some $j$ and $p$. If we pick up the cards so as to perform an out $m$-shuffle, then $(p-1) + (j-1)m$ cards will be picked up {\em before} this card. So then, the index of this card when the out $m$-shuffle is finished will be: $(p-1) + (j-1)m$. 

Note that before the cards were shuffled, the card lying in row $j$ and column $p$ in the above array had the following original index:
 \begin{align}\label{index}
i = (p-1)n + j - 1.
\end{align}
Careful arithmetic using Equation~\ref{index} yields  $(p-1) + (j-1)m \equiv mi \pmod{mn-1}$ for $i < mn-1$, which proves the desired relationship in Part (1) of the lemma.

Observe, then, that performing $O_m$ repeatedly $r$  times will move a card in position $i$ to the position $m^ri \pmod{mn-1}$ for any $i < mn-1$. Therefore, each card will be back in its original position exactly when $m^r \equiv 1\pmod{mn-1}$. Hence the order of the shuffle is precisely the order of $m$ in $\mathbb{Z} _{mn-1},$ concluding the proof of Part (1) of the lemma.

Next we consider {\em in $m$-shuffles}.  As before, place the cards (indexed 0 to $mn-1$) into $m$ stacks, as indicated in Array~\ref{array}. If we do an in $m$-shuffle, then $(m-p) + (j-1)m$ cards will be picked up {\em before} the card that lies in row $j$ and column $p$. So then the index of the card after the shuffle is complete will be: $(m-p) + (j-1)m$. If the card's original index was $i$, careful arithmetic using Equation~\ref{index} yields  $$(m-p) + (j-1)m \equiv mi + m-1 \pmod{mn+1},$$ as desired.

Using this formula, we see that performing $I_m$ a sequence of $\ell$ times will move a card in position $i$ to  position $m^\ell i + m^\ell - 1 \pmod{mn+1}.$ Therefore each card will be back in its original position after $\ell$ shuffles exactly when $i \equiv m^\ell i + m^\ell - 1 \pmod{mn+1}$ for all $i$. This will be the case if and only if $m^\ell \equiv 1 \pmod{mn+1}$. Part (2) of the lemma follows, as before.
\end{proof}

With these basic facts in place, we turn our focus to the special case where the size of the deck of cards is a power of $m$, say, $m^k$.  
As before, we index each card in the deck by its distance from the top of the deck from $0$ to $m^k-1$. The index $i$ on each card can be naturally expressed in base-$m$ as a $k$-tuple of digits: $(x_1,x_2,\ldots,x_k)$, where
$$i =x_1m^{k-1} + x_2m^{k-2} + \cdots + x_{k-1}m + x_k.$$

In Lemma~\ref{out}, we discussed the effect of in and out $m$-shuffles on a card's index, but now we describe $m$-shuffles' effects on the base-$m$ index of a card.

\begin{lemma}\label{basem}
Consider a deck of $m^k$ cards. Index each card in the deck by its distance from the top of the deck from $0$ to $m^k-1$, and express this index in base-$m$ as a $k$-tuple $(x_1,x_2,\ldots,x_k)$. Then $m$-shuffles send a card with index $(x_1,x_2,\ldots,x_k)$ to a new position as follows:
\begin{align}
O_m: (x_1,x_2,\ldots,x_k) &\longrightarrow (x_2,x_3,\ldots,x_k,x_1)\label{out-formula}\\
I_m: (x_1,x_2,\ldots,x_k) &\longrightarrow (x_2,x_3,\ldots,x_k,\overline{x}_1),\label{in-formula}
\end{align}
where $\overline{x}_1 = (m-1)-x_1.$ Moreover, the shuffles $O_m$ and $I_m$ have order $k$ and $2k$, respectively.
\end{lemma}

\begin{remark}
We will refer to the operation $\overline{x_i} = (m-1)-x_i$ as {\em flipping} the entry $x_i$. Observe that the flipping operation is of order two. That is, $\overline{\overline{x_i}} = x_i$.
\end{remark}

\begin{proof}
This lemma is an exercise in applying Lemma~\ref{out} to the special case of a deck of $m^k$ cards. 

Lemma~\ref{out} tells us that the cards with index 0 or $m^k-1$ are fixed by an out $m$-shuffle $O_m$, and a card with index $i < m^k-1$ is moved to index $mi \pmod{m^k-1}$. We compute:
\begin{align*}
mi & = m(x_1m^{k-1} + x_2m^{k-2} + \cdots + x_{k-1}m + x_k)\\
& = x_1m^{k} + x_2m^{k-1} + \cdots + x_{k-1}m^2 + x_km\\
& \equiv x_2m^{k-1} + \cdots + x_{k-1}m^2 + x_km + x_1 \pmod{m^k - 1}
\end{align*}
Therefore, it follows that for all $i$, the out $m$-shuffle has the following effect on a card's base-$m$ index: $O_m: (x_1,x_2,\ldots,x_k) \longrightarrow (x_2,x_3,\ldots,x_k,x_1)$, as desired. Clearly, the shuffle must be repeated $k$ times in order for each $k$-tuple to return to its original form, so $O_m$ has order $k$.

Now let us move to in $m$-shuffles $I_m$. Lemma~\ref{out} tells us that a card with index $i$ is moved to index $mi + (m-1) \pmod{m^k+1}$. We compute this value:
\begin{align*}
mi + (m-1) & = m(x_1m^{k-1} + x_2m^{k-2} + \cdots + x_{k-1}m + x_k) + (m-1)\\
& = x_1m^{k} + x_2m^{k-1} + \cdots + x_{k-1}m^2 + x_km + (m-1)\\
& \equiv x_2m^{k-1} + \cdots + x_{k-1}m^2 + x_km + (m-1) -  x_1 \pmod{m^k + 1}
\end{align*}
Hence, the in $m$-shuffle $I_m$ has the following effect on a card's base-$m$ index:\\
$I_m: (x_1,x_2,\ldots,x_k) \longrightarrow (x_2,x_3,\ldots,x_k,\overline{x}_1),$  where $\overline{x}_1 = (m-1)-x_1.$  The shuffle clearly must be repeated $2k$ times in order for each $k$-tuple to return to its original form, so $I_m$ has order $2k$.

 \end{proof}

With a deck of $m^k$ cards, we can also consider $m^2$-shuffles, $m^3$-shuffles, and so on. But observe that an $m^y$-shuffle for any positive $y<k$ is equivalent to repeating an $m$-shuffle $y$ times. In light of Equations~\ref{out-formula} and~\ref{in-formula}, it follows that $m^y$-shuffles have the following effect on the base-$m$ expansion of the card's index:
 \begin{align}
O_{m^y}: (x_1,x_2,\ldots,x_k) &\longrightarrow (x_{y+1},\ldots,x_k,x_1,\ldots,x_{y})\label{myshuffles}\\
 I_{m^y}: (x_1,x_2,\ldots,x_k) &\longrightarrow (x_{y+1},\ldots,x_k,\overline{x_1},\ldots,\overline{x_{y}})\label{myshuffles2}
 \end{align}
 \noindent where $\overline{x_i} = (m-1)-x_i$.\\
 
 
 \section{Generalized perfect shuffle groups for decks of size $m^k$}~\label{results}
At last, we are ready to determine the structure of the group of $m^y$-shuffles on a deck of $m^k$ cards for all possible $m$, $k$, and $y$. Alongside these discoveries, we give concrete examples of these groups. The tools used here are all standard group theory techniques. 

We proceed in three cases, each of which is proved separately: Case (1) $y$ is odd and relatively prime to $k$ (Theorem~\ref{odd}), Case (2) $y$ is even and relatively prime to $k$ (Theorem~\ref{even}), and Case (3) $y$ and $k$ are not relatively prime (Corollary~\ref{corollary}). Together, these three results prove Theorem~\ref{mainresult} stated in the introduction.\\

\begin{theorem}~\label{odd} Consider a deck of $m^k$ cards for some $m, k>1$. Let $y$ be a positive odd integer such that $y<k$ and $y$ is relatively prime to $k$. The $m^y$-shuffle group is as follows: $\langle I_{m^y},O_{m^y} \rangle \cong \mathbb{Z}^k_2 \rtimes_\phi \mathbb{Z}_k$, where $\mathbb{Z}_k$ acts by a cyclic shift on $\mathbb{Z}^k_2$. That is, $$\phi(1)\cdot(a_1, a_2, \ldots, a_k) = (a_k, a_1, a_2, \ldots, a_{k-1}).$$
\end{theorem}

\begin{proof}
We first solve the problem in the special case $y = 1$. Then we will show that if $y > 1$, we can reduce the situation back down to this special case. 

Consider $m$-shuffles $I_m$ and $O_m$ acting on a deck of $m^k$ cards. Using these two shuffles, we create the following set of shuffles:
$$\{B_j = O_{m}^{j-1} I_{m} O_{m}^{-j} ~| ~j = 1, 2, \ldots k\}.$$ 
(Note that we read the multiplication of shuffles from left to right.) 

Now let's consider the effect of each of the shuffles $B_j$ on a card's base-$m$ index. In light of Equations~\ref{out-formula} and~\ref{in-formula}, the shuffle $B_1 = I_{m} O_{m}^{-1}$ flips the first entry of a card's base-$m$ index and makes no other changes. The shuffle $B_2 = O_m I_mO_m^{-2} $ only flips the second entry in the card's index. This pattern continues for all $B_j$. For example, when $k=5$, we have:
 \begin{align*}
 B_1:(x_1,x_2, x_3, x_4, x_5) &\longrightarrow (\overline{x_1}, x_2, x_3, x_4, x_5)\\
 B_2:(x_1,x_2, x_3, x_4, x_5)&\longrightarrow  (x_1, \overline{x_2}, x_3, x_4, x_5)\\
 B_3:(x_1,x_2, x_3, x_4, x_5)&\longrightarrow  (x_1, x_2, \overline{x_3}, x_4, x_5)\\
 B_4:(x_1,x_2, x_3, x_4, x_5) &\longrightarrow (x_1, x_2, x_3, \overline{x_4}, x_5)\\
 B_5:(x_1,x_2, x_3, x_4, x_5) &\longrightarrow  (x_1, x_2, x_3, x_4, \overline{x_5})
\end{align*}

Observe that $B_iB_j = B_jB_i$ for all $i, j$. Hence the group $\langle B_1, B_2,\ldots, B_k \rangle$ is abelian. Moreover every element in the group has order 2, since $\overline{\overline{x_i}} = x_i$. Finally, note that the generating set for the group $\langle B_1, B_2,\ldots, B_k \rangle$ cannot be reduced in size, since each shuffle changes a disjoint part of a card's base-$m$ index. Thus it follows that $\langle B_1, B_2,\ldots, B_k \rangle \cong \mathbb{Z}_2^k$.

This group $\langle B_1, B_2,\ldots, B_k \rangle$ intersects the group generated by the out $m$-shuffle $\langle O_m \rangle$ only at the identity shuffle. Moreover, the product of the two groups contains both $O_m$ and $I_m$ and hence generate the whole group $\langle I_m, O_m \rangle$. Finally, observe that $\langle O_m \rangle$ has order $k$, and $O_m$ acts on the shuffles $B_j$ via a cyclic shift: 
 \begin{align*}
 O_mB_jO_m^{-1} = 
 \begin{cases}
 B_{j+1}  &\textrm{ if } j = 1,\ldots,k-1\\
 B_{1} &\textrm{ if } j = k.
 \end{cases}
 \end{align*}
Therefore we conclude that the $m$-shuffle group on $m^k$ cards $\langle I_m,O_m \rangle$  is isomorphic to $\mathbb{Z}^k_2 \rtimes \mathbb{Z}_k$, where the generator of $\mathbb{Z}_k$ acts on $\mathbb{Z}_2^k$ by a cyclic shift. So we have proved the theorem in the special case $y = 1$. 

Now we consider the situation of doing $m^y$-shuffles on $m^k$ cards, where $y>1$ is odd and relatively prime to $k$. As we noted in the previous section, the shuffles $O_{m^y}$ and $I_{m^y}$ can be expressed as repeated application of $m$-shuffles:
$$O_{m^y} = (O_m)^y  \hspace{.5cm} \text{and} \hspace{.5cm}  I_{m^y} = (I_m)^y.$$
Moreover, recall that the $m$-shuffles $O_m$ and $I_m$ have orders $k$ and $2k$, respectively. Since $y$ is relatively prime to both of these values, it follows that 
$$\langle O_{m^y} \rangle = \langle~ (O_m)^y ~\rangle = \langle O_m \rangle $$ and  $$\langle I_{m^y} \rangle = \langle~ (I_m)^y ~\rangle = \langle I_m \rangle.$$ 
Hence we have $\langle O_{m^y}, I_{m^y} \rangle = \langle I_m,O_m \rangle$, which proves that this group  $\langle O_{m^y}, I_{m^y} \rangle$ has the same structure as the case $y=1$, as desired. 

\end{proof}

Lest we lose the forest for the trees, let's look at an example that illustrates the theorem we have just proved. Suppose we have $m^2$ cards, for some natural number $m>1$. Theorem~\ref{odd} tells us that we are only able to reach a total of 8 card orderings with $m$-shuffles on these $m^2$ cards. The shuffle group is $(\mathbb{Z}_2 \times \mathbb{Z}_2) \rtimes \mathbb{Z}_2$. (This group is more commonly known as the dihedral group of the square, $D_4$.) The associated Cayley graph illustrates the relationship among the shuffles. In this case, the graph is a cube, as shown in Figure~\ref{m^2Cards}. Observe that the group is no different from the special case of 2-shuffles on a deck of 4 cards which we discussed earlier. Indeed, the group structure is independent of the value of $m$. 

\begin{figure}
\begin{center}
\includegraphics[width=9cm]{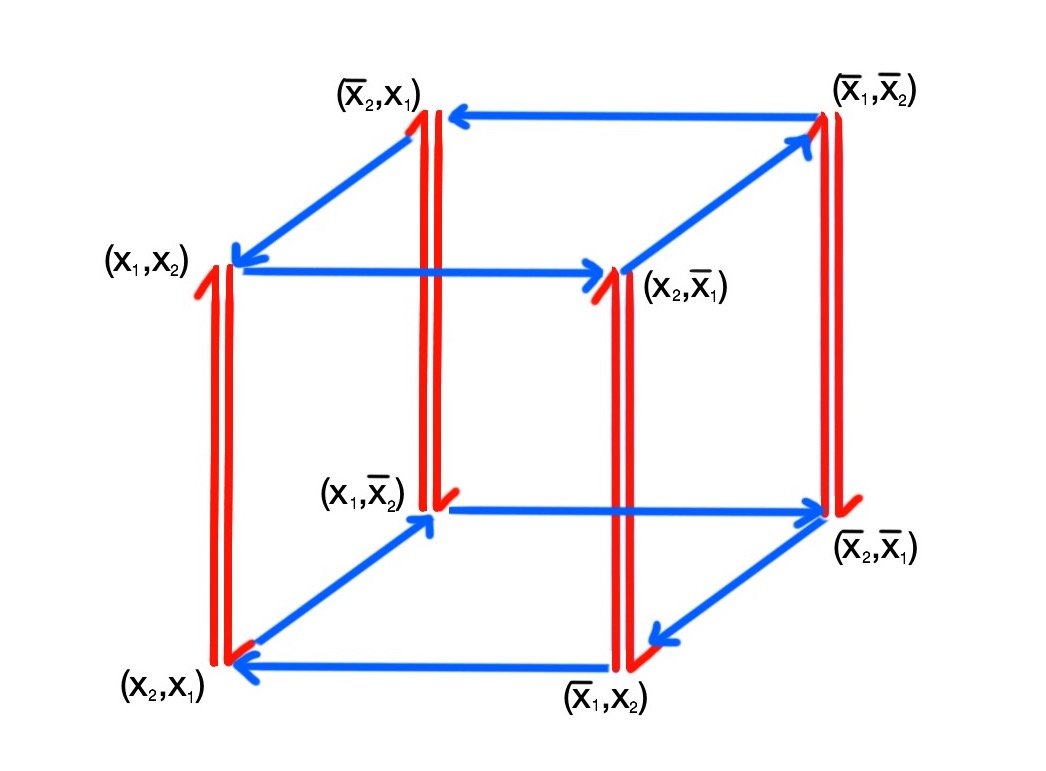}
\end{center}
\caption{A deck of $m^2$ cards can be shuffled with out $m$-shuffles (red) and in $m$-shuffles (blue) to reach eight different possible orderings total. Before shuffling, each card's position can be expressed in base $m$ as $(x_1, x_2)$, and after any $m$-shuffle, each card's position can be described in this format as well, where $\overline{x_i} = m - x_i$.
 }\label{m^2Cards}
\end{figure}

We now consider the case where $y$ is even and relatively prime to $k$. The theorem and proof are reminiscent of the previous case when $y$ is odd, with a few necessary twists. 

\begin{theorem}\label{even}
Consider a deck of $m^k$ cards for some $m, k>1$. Let $y$ be a positive even integer with $y<k$ and $y$ is relatively prime to $k$. Then the $m^y$-shuffle group is as follows: $\langle I_{m^y},O_{m^y} \rangle \cong \mathbb{Z}^{k-1}_2 \rtimes_\phi \mathbb{Z}_k$, where the action of $\mathbb{Z}_k$ on $\mathbb{Z}^{k-1}_2$ is given by:
$$\phi(1)\cdot(a_1, a_2, \ldots, a_{k-1}) = (a_{k-1}, ~~a_1+ a_{k-1}, ~~a_2+ a_{k-1},~ \ldots, ~~a_{k-2} + a_{k-1}).$$
\end{theorem}

\begin{proof}
To begin, we make some simplifying observations. Recall that since the out $m$-shuffle $O_m$ has order $k$ and since $y$ is relatively prime to $k$, it follows that 
$$\langle O_{m^y} \rangle = \langle~ (O_m)^y ~\rangle = \langle O_m \rangle.$$ 
And furthermore, since the in $m$-shuffle $I_m$ has order $2k$ and $\gcd(y, 2k) = 2$, we observe
$$\langle I_{m^y} \rangle = \langle~ (I_m)^y ~\rangle = \langle (I_m)^2 \rangle = \langle I_{m^2} \rangle.$$ 
Hence, $\langle  I_{m^y}, O_{m^y} \rangle = \langle I_{m^2}, O_m \rangle$. Therefore, it suffices to show that the group $\langle I_{m^2}, O_m \rangle$ has the structure described above in the statement of the theorem. 

To accomplish this, we create a new set of shuffles using $I_{m^2}$ and $O_m$. Consider the following: 
$$\{C_j = O_m^{j-1} I_{m^2} O_m^{-(j+1)}  ~| ~j = 1, 2, \ldots, k\}.$$ 
Despite their complicated appearance at first glance, we can easily describe the effect of these shuffles $C_j$ on the deck of cards. Using Equations~\ref{myshuffles} and~\ref{myshuffles2}, we observe that the shuffle $C_1 = I_{m^2} O_m^{-2}$ flips the first 2 entries of a card's base-$m$ index ($x_1$ and $x_2$), leaving everything else unchanged. The shuffle $C_2 = O_mI_{m^2} O_m^{-3}$ only flips the entries $x_2$ and $x_3$. This pattern continues, with $C_j$ only flipping the entries $x_j$ and $x_{j+1}$ for $j = 1, 2, \ldots k-1$. The final shuffle $C_k$ flips the entries $x_1$ and $x_k$. For example, in the case of $k = 5$, we have:
 \begin{align*}
 C_1:(x_1,x_2, x_3, x_4, x_5) &\longrightarrow (\overline{x_1}, \overline{x_2}, x_3, x_4, x_5)\\
 C_2:(x_1,x_2, x_3, x_4, x_5)&\longrightarrow  (x_1, \overline{x_2}, \overline{x_3}, x_4, x_5)\\
 C_3:(x_1,x_2, x_3, x_4, x_5)&\longrightarrow  (x_1, x_2, \overline{x_3}, \overline{x_4}, x_5)\\
 C_4:(x_1,x_2, x_3, x_4, x_5) &\longrightarrow (x_1, x_2, x_3, \overline{x_4}, \overline{x_5})\\
C_5:(x_1,x_2, x_3, x_4, x_5) &\longrightarrow (\overline{x_1}, x_2, x_3, x_4, \overline{x_5})
\end{align*}
These shuffles $C_j$ are all order 2 and commute with each other, so it follows that $\langle C_1, C_2, \ldots, C_k\rangle \cong (\mathbb{Z}_2)^d$, for some $d \leq k$.  

Consider the product of all of these shuffles:  $C_1C_2C_3 \cdots C_k$. This shuffle flips all entries in each card's base-$m$ index exactly 2 times. The net result, then, is that the shuffle returns every card to its original position. Therefore, we have $C_1C_2\cdots C_k = Id$, where $Id$ denotes the identity shuffle. We can rearrange this relation as: $C_k = C_1C_2\cdots C_{k-1}$. Hence the generators of the subgroup $\langle C_1, C_2, \ldots, C_k\rangle $ can be reduced: $\langle C_1, C_2, \ldots, C_k\rangle  \cong \langle C_1, C_2, \ldots, C_{k-1} \rangle $. 

We now claim that the generating set cannot be further reduced in size. That is, we claim that there are no nontrivial relations among $C_1, C_2, \ldots, C_{k-1}$. Since these shuffles $C_i$ all commute and have order 2, the only possible further nontrivial relation would be a product of a subset of $\{C_1, C_2, \ldots, C_{k-1}\}$, but any such product cannot be trivial because it will leave at least one entry $x_i$ flipped. Hence, we conclude that $\langle C_1, C_2, \ldots, C_{k-1}\rangle \cong (\mathbb{Z}_2)^{k-1}$

The group $\langle C_1, C_2, \ldots, C_{k-1}\rangle$ intersects the group generated by the out $m$-shuffle $\langle O_m \rangle$ only at the identity, and the elements of the two groups together generate both $O_m$ and $I_{m^2}$. Moreover, observe that $\langle O_m \rangle$ has order $k$, and $O_m$ acts on the shuffles $C_j$ as follows:
 \begin{align}~\label{cyclicshift}
 O_mC_jO_m^{-1} = \begin{cases}
 C_{j+1}  &\textrm{ if } j = 1,\ldots,k-2\\
 C_k = C_1C_2\cdots C_{k-1} &\textrm{ if } j = k-1
 \end{cases}
 \end{align}
Thus, $\langle I_{m^2},O_m \rangle$ is isomorphic to the group $\mathbb{Z}_2^{k-1} \rtimes_\phi  \mathbb{Z}_k$, where $\phi$ is given by 
$$\phi(1)\cdot(a_1, a_2, \ldots, a_{k-1}) = (a_{k-1}, ~a_1+ a_{k-1}, ~a_2+ a_{k-1},~ \ldots, ~a_{k-2} + a_{k-1}),$$
as desired.
\end{proof}

As an example of the above theorem, suppose that we have $m^3$ cards for some natural number $m>1$. If we do in and out $m^2$-shuffles on the deck, we are only able to reach 12 card orderings regardless of the value of $m$. By Theorem~\ref{even}, the shuffle group is $(\mathbb{Z}_2 \times \mathbb{Z}_2) \rtimes_\phi \mathbb{Z}_3$, where $\phi(1)(a_1, a_2) = (a_2, a_1 + a_2)$. This group is isomorphic to the alternating group $A_4$.  The associated Cayley graph takes the shape of the cuboctahedron as in Figure~\ref{m^3Cards}.

\begin{figure}
\begin{center}
\includegraphics[width=9cm]{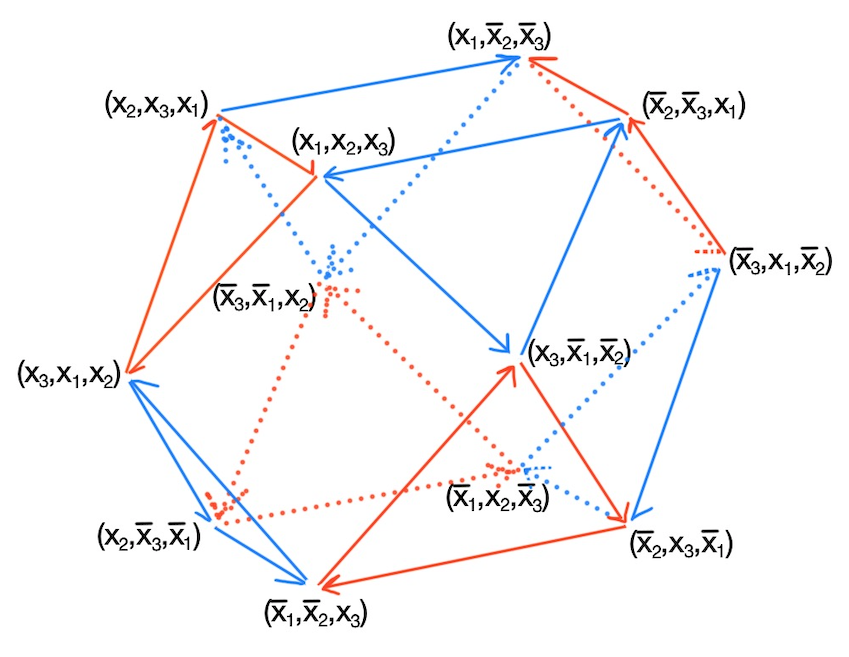}
\end{center}
\caption{A deck of $m^3$ cards can be shuffled with out $m^2$-shuffles (red) and in $m^2$-shuffles (blue) to reach twelve different possible orderings total. Before shuffling, each card's position can be expressed in base $m$ as $(x_1, x_2, x_3)$, and after any $m^2$-shuffle, each card's position can be described in this format as well, where we write $\overline{x_i}$ to denote $m - x_i$.}\label{m^3Cards}
\end{figure}

We conclude by considering the case where $k$ and $y$ are {\em not} relatively prime. This case follows as a corollary of the previous two results and, together with Theorems~\ref{odd} and~\ref{even}, completes the proof of Theorem~\ref{mainresult}.\\

\begin{corollary}~\label{corollary}
Consider a deck of $m^k$ cards for some $m, k>1$. Let $y$ be a positive integer with $y<k$ and $\gcd[y,k] = c>1$. The group generated by in and out $m^y$-shuffles on this deck $ \langle I_{m^y},O_{m^y} \rangle$ can be described as follows.

\begin{enumerate} 
\item If $y/c$ is odd, then the shuffle group is isomorphic to $\mathbb{Z}^{k/c}_2 \rtimes \mathbb{Z}_{k/c}$, where the action of $\mathbb{Z}_{k/c}$ is a cyclic shift as in Theorem~\ref{odd}.

\item If $y/c$ is even, then the shuffle group is isomorphic to $\mathbb{Z}^{\frac{k}{c}-1}_2 \rtimes \mathbb{Z}_{k/c}$, where the action of $\mathbb{Z}_{k/c}$ is as in Theorem~\ref{even}.
\end{enumerate}
\end{corollary}

\begin{proof}

Previously, we had considered each card's index in terms of its base-$m$ expansion. But now, we will consider each card's position in terms of its expansion in base $m^c$. As such, the card indices (which range from 0 to $m^k - 1$) can be expressed as a $(k/c)$-tuple in base $m^c$:  $(x_1,x_2, \ldots, x_{k/c})$.  

Recall that performing an $m^y$-shuffle is the same as performing an $m$-shuffle $y$ times. Equivalently, an $m^y$ shuffle can be accomplished by doing an $m^c$-shuffle $y/c$ times. Using this and the formulas in Lemma~\ref{basem}, the effect of perfect $m^y$-shuffles on a card's index written in base-$m^c$ is given as follows:
\begin{align*}
O_{m^y}: (x_1,x_2, \ldots, x_{k/c}) &\longrightarrow (x_{\frac{y}{c}+1},\ldots, x_{k/c}, x_1,x_2,\ldots ,x_{y/c})\\
I_{m^y}: (x_1,x_2, \ldots, x_{k/c}) &\longrightarrow (x_{\frac{y}{c}+1},\ldots, x_{k/c},\overline{x_1},\overline{x_2}, \ldots,\overline{x_{y/c}}),
\end{align*} 
 \noindent where $\overline{x_i} = (m^c-1)-x_i$.

Notice that the above two equations are identical to the equations for  of performing perfect $m^{y/c}$-shuffles on a deck of $m^{k/c}$ cards, as in Equations~\ref{myshuffles} and~\ref{myshuffles2}. Thus in this case, the structure of $m^y$-shuffle group on $m^k$ cards is exactly the same as that of $m^{y/c}$-shuffles on $m^{k/c}$ cards. And, since $y/c$ and $k/c$ are relatively prime, we can describe this group structure using our previous results.

Therefore, if $y/c$ is odd, then $\langle I_{m^y},O_{m^y} \rangle \cong \mathbb{Z}_2^{k/c} \rtimes \mathbb{Z}_{k/c}$, where the action of $\mathbb{Z}_{k/c}$ is a cyclic shift as in Theorem~\ref{odd}. If $y/c$ is even, then $\langle I_{m^y},O_{m^y} \rangle \cong \mathbb{Z}^{(k/c)-1}_2 \rtimes \mathbb{Z}_{k/c}$, where the action of $\mathbb{Z}_{k/c}$ is the action described in Theorem~\ref{even}.
\end{proof}

\begin{figure}
\begin{center}
    \begin{tabular}{| c | c | c |}
    \hline
    Size of deck & $m$ & Order of $\langle I_m,O_m \rangle$ \\ \hline
    4 & 2 & 8 \\ \hline \hline
    6 & 2 & 24 \\ \hline
    6 & 3 & 48 \\ \hline \hline
    8 & 2 & 24 \\ \hline
    8 & 4 & 12 \\ \hline \hline
    9 & 3 & 8 \\ \hline \hline
    10 & 2 & 1920 \\ \hline
    10 & 5 & 960 \\ \hline \hline
    12 & 2 & 7680 \\ \hline
    12 & 3 & 60 \\ \hline
    12 & 4 & 120 \\ \hline
    12 & 6 & 7680 \\ \hline \hline
    14 & 2 & 322560 \\ \hline
    14 & 7 & 645120 \\ \hline \hline
    15 & 3 & 384 \\ \hline
    15 & 5 & 384 \\ \hline \hline
    16 & 2 & 64 \\ \hline
    16 & 4 & 8 \\ \hline
    16 & 8 & 64 \\ \hline
    \end{tabular}
    \end{center}
 \caption{The orders of $m$-shuffle groups $\langle I_m,O_m \rangle$ for small deck sizes.}\label{Table}
 \end{figure}

\section{Parting observations}\label{future}
A natural next step in this investigation, of course, is to understand the situation for deck sizes which are {\em not} equal to $m^k$. We wrote a computer program to determine the order of $m$-shuffle groups for small deck sizes. See Figure~\ref{Table} for this information.

Another natural question returns to the idea of symmetry and shuffle groups. As mentioned previously, the standard perfect shuffle group $\langle I, O \rangle$ for a deck of $2n$ cards cannot be larger than $2^nn!$, since all elements in the group preserve central symmetry. Similarly, the $m$-shuffle groups $\langle I_m, O_m \rangle$ on a deck of $2n$ (or $2n+1$) cards also preserve central symmetry and hence can be no larger than $2^nn!$, as well~\cite{Medvedoff}. In their work on standard shuffles, Diaconis et al. described precisely the values of $n$ for which the standard shuffle group on $2n$ cards reaches this maximum value~\cite{Diaconis}. We ask whether such a result can be generalized? For example, consider a deck of $6$ cards. The maximum order possible is $48$, which is achieved with the 3-shuffle group. With a deck of 12 cards, however, the maximum order is 46,080, which is not realized by any $m$-shuffle group. What is the pattern?  \\

{\textbf{Acknowledgements: }
We thank Kent Morrison for his insightful comments on an early draft of this paper.}

\bibliographystyle{plain}
\bibliography{bibliography}

\end{document}